\documentclass[pdflatex,sn-mathphys-num]{sn-jnl}

\usepackage{svg}%
\usepackage{graphicx}%
\usepackage{multirow}%
\usepackage{amsmath,amssymb,amsfonts}%
\usepackage{amsthm}%
\usepackage{mathrsfs}%
\usepackage[title]{appendix}%
\usepackage{xcolor}%
\usepackage{textcomp}%
\usepackage{manyfoot}%
\usepackage{booktabs}%
\usepackage{algorithm}%
\usepackage{algpseudocode}%
\usepackage{listings}%
\usepackage{comment}
\usepackage{enumitem}
\usepackage{hyperref}
\usepackage{transparent}

\DeclareMathOperator{\conv}{Co}
\DeclareMathOperator{\inte}{int}
\DeclareMathOperator{\dom}{dom}

\newcommand{\GAP}{\mbox{\sc gap}}
\newcommand{\myphi}{\zeta}

\theoremstyle{thmstyleone}%
\newtheorem{theorem}{Theorem}
\newtheorem{proposition}[theorem]{Proposition}%
\newtheorem{lemma}[theorem]{Lemma}%
\newtheorem{corollary}[theorem]{Corollary}%
\theoremstyle{thmstyletwo}%
\newtheorem{remark}{Remark}%

\theoremstyle{thmstylethree}%

\begin{document}

\title[Article Title]{A Dantzig-Wolfe Decomposition Method for Quasi-Variational Inequalities}


\author[1]{\fnm{Manoel} \sur{Jardim}}\email{manoel.jardim@impa.br}

\author[2]{\fnm{Claudia} 
\sur{Sagastizábal}}\email{sagastiz@unicamp.br}

\author*[1]{\fnm{Mikhail} \sur{Solodov}}\email{solodov@impa.br}

\affil[1]{\orgname{IMPA}, \orgaddress{\street{Estrada Dona Castorina 110}, \city{Rio de Janeiro}, \postcode{22460-320}, \state{RJ}, \country{Brasil}}}

\affil[2]{\orgdiv{IMECC}, \orgname{Unicamp}, \orgaddress{\street{Rua Sérgio Buarque de Holanda}, \city{Campinas}, \postcode{13083-859}, \state{SP}, \country{Brasil}}}


\abstract{We propose an algorithm to solve quasi-variational inequality problems,
based on the Dantzig-Wolfe decomposition paradigm. 
Our approach solves in the subproblems variational inequalities, which is a simpler problem,  
while restricting quasi-variational inequalities in the master subproblems, making them
generally much smaller in size when the original problem is large-scale. 
We prove global convergence of our algorithm, assuming that the mapping of the 
quasi-variational inequality is either single-valued and continuous
or it is set-valued maximally monotone.
 Quasi-variational inequalities serve as a framework for several 
equilibrium problems, and we apply our algorithm to  an important example 
in the field of economics, namely the Walrasian equilibrium problem formulated as a 
generalized Nash equilibrium problem. 
Our numerical assessment demonstrates the usefullness
of the approach 
for large-scale cases, where thanks to decomposition it can solve problems not tractable directly
by state-of-the-art software.
}




\maketitle
\newpage
\section{Introduction}
The framework provided by a quasi-variational inequality (QVI) 
setting encompasses many problems related to 
optimization and equilibrium; see \cite{2004}.
The goal of a QVI problem is to find  a pair
$(x^*,z^*)$ such that
\begin{equation}\label{qVI}
x^* \in K(x^*) \mbox{ with } z^* \in F(x^*) \mbox{ satisfy }
\left< z^*, y-x^* \right> \geq 0\; \mbox{ for all }y \in K(x^*).
\end{equation}
In these relations, $\left< \cdot , \cdot \right>$ stands for
the Euclidean inner product in $\mathbb{R}^n$ and
$F:\mathbb{R}^n \rightrightarrows \mathbb{R}^n$ 
and $K:\mathbb{R}^n \rightrightarrows \mathbb{R}^n$ are two
maps associating vectors in $\mathbb{R}^n$ to subsets of $\mathbb{R}^n$.
A variational inequality (VI) corresponds to
a less general case, with 
$K$ in \eqref{qVI} being a constant set-valued map, i.e., 
$K(x)=D\subset \mathbb{R}^n$ for all $x\in \mathbb{R}^n$.
In most meaningful VI settings, in particular those subsuming 
local optimality conditions in optimization, the set $D$ is convex.
If $D$ is further the nonnegative orthant, the corresponding
VI is a
nonlinear complementarity problem.
The so-called mixed complementarity problem arises when
the set $D$ is a generalized box (given component-wise by one-sided or two-sided
bound constraints, or no constraints at all for some components). 
 For details and related discussions, we refer again to \cite{2004}.  

The introduction in \cite{Kanzow_2019} provides a good review of the history of QVIs. We recall 
here that the first work on the subject dates back to 
\cite{Bensoussan1973}, where impulsive control problems were formulated according to the format \eqref{qVI}. 
The framework has proven effective in handling important applications in engineering 
\cite{Outrata_1998,Kravchuk_2007}, 
transportation \cite{SCRIMALI_2004}, and economics \cite{Harker_1991}. 
In particular, QVIs offer a favorable environment for modeling equilibria, 
of the Radner type as in \cite{Aussel_2021}, or resulting from a generalized 
Nash game \cite{Facchinei_2007}.

Much of QVI literature focuses mainly on theoretical results, 
especially concerning the existence of solutions. 
Algorithmic research on the subject is less developed, probably due to 
certain inherent difficulties associated to QVIs. 
The proposal in \cite{Fukushima_2007}, to solve \eqref{qVI} by
minimizing a nonsmooth gap function, does not present a
specific solution algorithm.
Inspired by the structure of generalized Nash games,
\cite{Pang_2005} solves sequentially VIs 
that are shown to converge to a solution of \eqref{qVI}.
The scheme was revisited and enhanced in
\cite{Facchinei_2010,Kanzow_2016,Kanzow_2017}. 
Newtonian approaches for solving
the system of optimality conditions derived from \eqref{qVI}
were considered in the works \cite{Hinterm_ller_2012,Outrata_1995}, 
reporting only some limited numerical experiments.
The specialized potential reduction interior point method proposed in
\cite{Facchinei_2013} exhibits good performance on
the test library in \cite{QVILIBAL}. 

It is, nonetheless, unclear how well the aforementioned algorithms scale for large instances.
In mathematical optimization, decomposition methods are 
best suited to handle large problems.
To provide a first step in this direction when dealing with QVIs,
we introduce a Dantzig-Wolfe-like method that is shown 
to converge to a solution of \eqref{qVI} under natural assumptions. 
 Borrowing the terminology for linear programs (LPs),
the proposed algorithm
alternates between solving a simple QVI as its ``master program''
and a VI as its ``subproblem''. 
The motivation is to tailor the method
 to facilitate subproblem separability,   
crucial when dimensions in \eqref{qVI} are large.
We also note the following interesting (and advantegeous!) feature
of our method, compared to Dantzig-Wolfe techniques in some previous settings.
For linear programming and variational inequalities,
the master program and subproblem to be solved at each iteration are
of the same complexity class (LPs and VIs respectively, though 
they are of course simpler than the original LP or VI, so that
the decomposition makes sense). By contrast,
in our Dantzig-Wolfe decomposition for QVIs, the  
subproblem to be solved in a given iteration actually belongs to the
{\em structurally simpler class} than the original problem: it is a
VI instead of a QVI (see the right box in Figure~\ref{fig-scheme}  below).

We mention, in passing, that sometimes QVIs can be reduced to a VI, 
as is the case of variational equilibria of some Nash games 
\cite{Facchinei_2007_,Kulkarni_2012}. 
For a discussion of related issues in the context of
energy markets we refer to \cite{lss:models}.
Furthermore, the reproducible set-valued maps introduced in 
 \cite{Aussel_2016} 
identify situations in which solving a VI provides all the solutions to a QVI. But 
these are very specific configurations, far from representing general situations. 
Indeed, QVIs have a significantly more complex structure than VIs, and therefore more sophisticated tools and developments are needed to solve problems such as \eqref{qVI} using decomposition techniques.

The Dantzig-Wolfe (DW) decomposition method was introduced in \cite{Dantzig_1961} 
to solve large linear programs having a structured feasible set, whose constraints 
can be separated into ``hard'' and ``easy'' ones (generally separable). 
Each iteration solves a master program followed by a subproblem, 
respectively aiming at guaranteeing feasibility and objective function decrease. 
The master program outputs multipliers associated with the hard constraints, which
define the subproblem objective function. The subproblem then has an ``easy'' feasible set, 
usually yielding further separable problems that can be solved in parallel.
The subproblem solution is informed to the master
program, so that in the next iteration the approximation of the master's feasible set is improved.

The DW decomposition for VIs was introduced in \cite{Fuller_2005,Chung_2010}.
The approach was extended and improved in \cite{Luna_2012}, including several
theoretical generalizations and an application to large-scale 
generalized Nash games.
Another use of DW techniques of \cite{Luna_2012} is
given in \cite{lss:risk}, where a class of risk-averse stochastic equilibrium problems 
is considered, numerically assessed on the ``real-world'' European market of natural gas. 
Benders decomposition for VIs, which is the dual 
approach to DW, is presented in \cite{lss:Benders}.

In this work, we show how
the DW approach can be applied to very general QVIs. 
When iterating between solving QVI-master problems and VI-subproblems sequentially, 
our DW method computes in the process a certain gap function that provides information about 
convergence.  
As illustrations, we start by computing the Walrasian equilibrium of 
large economies, a very classical and important problem in economics, 
 recently considered in \cite{deride,bss:decomp}. 
We afterwards solve also some 
academic QVI examples from \cite{QVILIBAL,Facchinei_2013}, said to have
``moving sets''.  In both cases,
our approach gives the same results as
the commercial solver GAMS \cite{GAMS}, with substantially less computational
effort for the larger dimensions. 
Over many random instances generated 
for the Walrasian equilibrium problem, the DW method computing times are not only shorter but consistently less
volatile than those for the direct solution by GAMS. 

The rest of the paper is organized as follows. 
In Section \ref{dwqvi}, we give the mathematical formulation of 
our DW algorithm.
Section \ref{convanalysis} is devoted to convergence results. 
When the problem has a certain special structure,
Section \ref{jacobi} gives details about techniques to uncouple variables and make subproblems separable.
Presenting the two sets of
numerical experiments (Walrasian equilibrium and moving set example),
the results in Section \ref{numericalresults} provide empirical evidence
of the good performance of our method on large-scale problems,
when compared with the direct application
of GAMS software.
 
Some final words about our notation and terminology.
The Euclidean inner product in $\mathbb{R}^n$ is denoted by
$\left< \cdot , \cdot\right>$ 
 and $\|\cdot\|$ is the associated norm.
By $\|\cdot\|_{\infty}$ we denote the maximum norm. For a set $D\subset \mathbb{R}^n$,
its interior is denoted by $\inte(D)$ and
its convex hull (the smallest convex set in $\mathbb{R}^n$
that contains $D$), is denoted by $\conv{D}$.
  For a convex set $D$, the notation
$\mathcal{N}_D (x) =\{u : \left< u , y-x\right> \le 0,\; \forall\, y\in D\}$
stands for the normal cone to $D$ at $x$ when $x\in D$ ($\mathcal{N}_D (x) =\emptyset$
otherwise).
The mapping $F:\mathbb{R}^n \rightrightarrows \mathbb{R}^n$ is strongly monotone
if there exists $c>0$ such that
$\left< u -v , x-y\right> \ge c\|x-y\|^2$ for all 
$x,y\in \dom(F)=\{z\in \mathbb{R}^n : F(z)\neq\emptyset\}$ and
all $u\in F(x)$, $v\in F(y)$.
Also, $F$ is monotone if the above inequality holds for $c=0$.
A monotone set-valued mapping $F:\mathbb{R}^n \rightrightarrows \mathbb{R}^n$
 is maximally monotone if its graph 
$\{(x,u)\in \mathbb{R}^n \times \mathbb{R}^n \,:\, 
x\in \dom(F),\, u\in F(x)\}$
is not properly contained in the graph of any other monotone mapping. 
We also use the properties that
a maximally monotone  operator 
 $F$ is both locally bounded in $\inte(\dom(F))$ \cite{Rockafellar_1969},
and outer-semicontinuous \cite[Chapter 4]{Burachik_2008}.

\section{Dantzig-Wolfe decomposition for QVI}\label{dwqvi}

Let there be
given functions $h:\mathbb{R}^n \to \mathbb{R}^l$ and
 $g(\cdot,\cdot):\mathbb{R}^n\times\mathbb{R}^n\to \mathbb{R}^m$,
with $h$ convex, and
$g(\cdot,x)$ convex and differentiable for each $x\in\mathbb{R}^n$,
and the set in \eqref{qVI} be given by
\begin{equation} \label{Kx}
K(x)=K_g(x)\cap K_h, 
\mbox{ where } \begin{array}{lcl}
K_g(x)&=&\left\{y\in \mathbb{R}^n: g(y,x)\leq 0\right\}\\
K_h&=&\left\{y\in \mathbb{R}^n: h(y)\leq 0\right\}\,.
\end{array}
\end{equation}

The QVI setting
is a good candidate for the DW technique,
as in \eqref{Kx} ``difficult'' constraints can be considered those
involving both variables $x$ and $y$, while the constraints in $y$ are 
naturally ``easy'' (or at least easier). 

We assume that $K_h \subset \inte(\dom(F))$,  
and that some $y^1\in K_g(y^1)\cap K_h$ is given, to start the process.

\subsection{General organization and master QVI definition}
In \eqref{Kx}, it seems easier to ensure feasibility with respect to $K_h$.
Since the $x$-parameterized constraints $g(y,x)$ are naturally harder to handle, 
they are dealt with in the master.
Similarly to the DW scheme in linear programming,
the corresponding multiplier associated to the parametrized constraint
 provides its Lagrangian relaxation, which
 is incorporated into the definition of the operator in the subproblem, whose
feasible set is $K_h$, and it is therefore a VI.

\begin{figure}[hbt]
\centering
\scalebox{0.5}{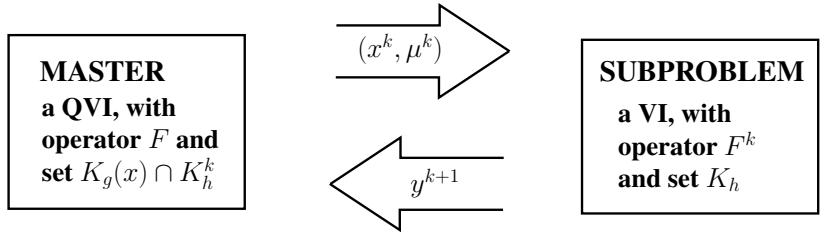}
\caption{Some elements of the DW decomposition for \eqref{qVI}-\eqref{Kx}. On the left, 
the master problem
solves a QVI that outputs $x^k$ and a multiplier $\mu^k$ associated to
the $g$-constraints. This primal-dual pair is used by the subproblem
on the right to define the operator $F^k$, and return a solution $y^{k+1}$ 
to the master problem. The output of the subproblem is used by the master in the next iteration, to define the set $K_h^{k+1}$.}\label{fig-scheme}
\end{figure}

To introduce gradually the notation, the scheme in Figure~\ref{fig-scheme} reports
the information exchanged between the master problem and 
the subproblem at each iteration $k$ of the process.
The information generated when solving the VI subproblem at iteration $k$,
which is $y^{k+1}$, is used by the master QVI at iteration $k+1$
to approximate the set $K_h$.
The approximation takes the convex hull of the output from 
the past VI-subproblems.

When formulating the master QVI solved
  at iteration $k$, the last subproblem information available 
to the master is $y^k$.
Accordingly, the QVI master \eqref{master} at iteration $k$ finds a pair $(x^{k},z_m^{k})$ for which
\begin{equation}\label{master}
\begin{array}{rlllll}
x^{k} &\in K_g(x^{k})\cap &K_h^{k},\; z_m^{k} \in F(x^{k}) 
&\mbox{satisfy}&
\left< z_m^{k}, x-x^{k} \right> \geq 0\\
&&& \mbox{ for all }&x \in 
 K_g(x^{k})\cap K_h^{k}\\
&\mbox{}&
\mbox{where }
K_h^{k}=\conv Y^{k}&\mbox{ and }&Y^{k}=\left\{y^1,\ldots,y^{k}\right\}\,.
\end{array}
\end{equation}
Another part of the output of the QVI master \eqref{master} is
a multiplier $\mu^k$ associated to
the $g$-constraints. 
Because the intersection with $\conv Y^{k}$
reduces the search for feasible elements to 
a $(k-1)$-dimensional simplex, the master 
problem \eqref{master} is a relatively simple QVI even if $n$ is very large.

\subsection{Subproblem definition}
As shown by Figure~\ref{fig-scheme}, after the master \eqref{master} is solved, the VI-subproblem receives 
in addition to $x^{k}$, information on
 $\mu^{k}$, the multiplier associated with the constraints defining $K_g(x^{k})$.
The pair $(x^k,\mu^k)$ is used to define the 
subproblem operator $F^k$, 
 for example as
\begin{equation}\label{def-Fk}
F^k(y) = F(y)+\sum\limits_{i=1}^m \mu_i^k\nabla_yg_i(x^k,x^k)\,.\end{equation}
Other choices are possible, depending on how the first and second terms,
respectively involving $F$ and $\nabla_y g$, are approximated.
For the latter, notice that \eqref{def-Fk} makes a constant approximation
of the gradient of the hard constraint.
The approximation could also let the first variable free or, to have a more
flexible framework, both options could be combined using
a parameter $\omega_i^k\in[0,1]$. It is then notationally
convenient to write the second term in \eqref{def-Fk} in matrix-vector form. 
Thus, we introduce a (transposed Jacobian) matrix $\Gamma^k(y)$ of order $n\times m$
with columns defined as below, for $j=1,\ldots,m$:
\begin{equation}\label{aprox-grad}
{\Gamma}_j^k(y)=\left\{\begin{array}{rlll}
&\nabla_yg_j(x^k,x^k)&\mbox{(constant)}\\
& \nabla_yg_j(y,x^k)&\mbox{(free)}\\
\omega_j^k&\nabla_yg_j(x^k,x^k)
+(1-\omega_j^k) \nabla_yg_j(y,x^k) &\mbox{(convex combination).}
\end{array}\right.
\end{equation}
Some comments regarding \eqref{aprox-grad} are in order.
When the product \(\Gamma^k(y)\mu^k\) is computed using
the (constant) first option above, it coincides with the expression for the second term in \eqref{def-Fk}.
The second option in \eqref{aprox-grad} can be advantageous when we aim to retain more information about 
 the $g$ constraints in the subproblem formulation. 
For our instances of Walrasian equilibrium in Section \ref{numericalresults},
the function $g$ depends linearly on $y$ (see \eqref{W-K}), and then all the options become the same.
In general though, this need not be the case, and having this feature adds
to the broader flexibility and applicability of the framework.
The impact of the third option is assesed in the numerical experiments on moving set problems
in Subsection~\ref{mov}.
Finally, notice that, although not made explicit in the notation,
the third option involves the parameter $\omega^k\in[0,1]^m$.

A similar approach can be employed for the first term in \eqref{def-Fk}, that is, 
approximating $F$ by an operator $\widehat F^k$ that
uses information output by the master QVI.
Along the lines in \cite{Luna_2012}, the approximation can be constant,
of first-order, or the actual mapping:
\begin{equation}\label{aprox}
\widehat{F}^k(y)=\left\{\begin{array}{lll}
\{z^k_m\}& \text{for } z_m^k \in F(x^k) \mbox{ from \eqref{master}}&\mbox{(constant)}\\
F(x^k)+\nabla F(x^k)(y-x^k),&\text{if $F\in C^1$ is single-valued}&\mbox{(first order)} \\
F(y)&&\mbox{(exact)}.\end{array}\right.
\end{equation}
Finally, if needed or advantageous,  
the VI-subproblem operator can also be regularized using
an $n\times n$ positive (semi)definite matrix
$Q^{k}$.

In full generality,
the VI subproblem \eqref{subp} at iteration $k$ finds a pair $(y^{k+1},z_s^{k+1})$ for which
\begin{equation}\label{subp}
\begin{array}{llll}
y^{k+1} \in K_h\,,&  z_s^{k+1} \in F^k(y^{k+1}) 
&\mbox{satisfy}&
\left< z_s^{k+1}, y-y^{k+1} \right> \geq 0\; \mbox{ for all }y \in K_h\\
&& \mbox{where}&
F^{k}(y)=\widehat{F}^k(y)+\Gamma^k(y)\mu^k +Q^{k}(y-x^k)\\
&&&\mbox{with $\widehat F^k$  from \eqref{aprox} and 
$\Gamma^k$ from \eqref{aprox-grad}.}
\end{array}
\end{equation}
For later use notice that, after solving \eqref{master}
and \eqref{subp}, the inclusions
\begin{equation} \label{atxk}
\begin{array}{rcl}
\myphi^k&=&z^k_m+\Gamma^k(x^k)\mu^k 
\in F^k(x^k)\mbox{ and }\\
\widehat\myphi^k&=&z^{k+1}_s-\Gamma^k(y^{k+1})\mu^k-Q^k(y^{k+1}-x^k)
\in \widehat F^k(y^{k+1})\end{array}
\end{equation}
hold for all the approximating variants proposed in \eqref{aprox-grad} and \eqref{aprox}.

Both QVI-master \eqref{master}
and VI-subproblem \eqref{subp} are assumed to be solvable at
all iterations.
We shall not go into extensive discussions on sufficient conditions for that assumption to hold, because
they are well known. 
For the VI subproblems,
in particular, solvability follows 
if the set $K_h$ is compact. 
Regardless of compactness of that set, 
subproblems are always solvable if 
the approximation $F^k$ in \eqref{subp} is strongly monotone, 
a property that can be ensured in our setting, as explained next. 

\begin{remark}[On strong monotonicity of subproblem operators]
\label{monotonicity} \em 

The approximating mapping $F^k$ in \eqref{subp} can always be chosen
to be strongly monotone.
To see this, note that $F^k$ involves three terms,
the first one being the approximation
$\widehat{F}^k$ defined in \eqref{aprox}. 
If $F$ is monotone, any option in \eqref{aprox} preserves
monotonicity. If
$F$ is not monotone, 
the constant option in \eqref{aprox} makes
the term $\widehat{F}^k$ monotone.
Regarding the second term defining $F^k$, it is always  
monotone because, by \eqref{aprox-grad}
 and by the convexity of $g_i(\cdot,x^k)$, it holds that
\begin{align*} 
\left<\Gamma^k(y)\mu^k-\Gamma^k(x)\mu^k,y-x\right> = \sum\limits_{i=1}^m \mu_i^k(1-\omega_i^k)\left<\nabla_yg_i(y,x^k)-\nabla_yg_i(x,x^k),y-x\right> \geq 0.
\end{align*}
The first two terms in $F^k$ are
monotone. Since in the third term the matrix 
$Q^k$ can be chosen positive definite matrix if needed,
strong monotonicity of $F^k$ follows.\qed
\end{remark}

\section{Algorithm statement and its convergence}\label{convanalysis}

After the master QVI is solved,
the information $x^k,z_m^k,\mu^k$ is available.
For $y\in\mathbb{R}^n$, we define the gap function
\begin{equation}
\GAP(y)=\left<\zeta^k, 
 y-x^k\right>\mbox{ with }
\zeta^k=z_m^k+\sum_{i=1}^m\mu_i^k\nabla_y g_i(x^k,x^k)
\in F^k(x^k) ,
\label{gap}
\end{equation}
where the inclusion $\zeta^k\in F^k(x^k)$ is by \eqref{atxk},
recalling also \eqref{aprox-grad} to evaluate $\zeta^k$. 

Convergence of our DW decomposition method is determined as in the 
linear programming setting, by monitoring the value
of the gap function at $y^{k+1}$, the subproblem solution. 
The framework of the method is outlined in Algorithm \ref{algo1}.
\begin{algorithm}[hbt]
\caption{Dantzig-Wolfe decomposition for QVI}\label{algo1}
\begin{algorithmic}
\Require $y^1 \in K_g(y^1)\cap K_h$.
\Ensure Accumulation points of the iterates solving QVI \eqref{qVI}-\eqref{Kx}.
\State Set $\GAP^1=-\infty$ and $k \leftarrow 1$.
\While{$\GAP^k<0$} 
\State {\sc master solution}: 
 solve \eqref{master} to compute the pair $(x^k, z_m^k)$ and 
the multiplier $\mu^k$.
\State {\sc subproblem solution}: 
 solve \eqref{subp} to compute $y^{k+1}$.
\State {\sc stopping criterion}: 
Compute $\GAP^{k+1}=\GAP(y^{k+1})$ defined in \eqref{gap}.
\State {\sc update}: 
$k \leftarrow k+1$
\EndWhile
\end{algorithmic}
\end{algorithm}

In Theorem~\ref{principal}
we shall prove that accumulation points of
the iterates generated by Algorithm~\ref{algo1} solve QVI \eqref{qVI}-\eqref{Kx}.

We first state            
the Karush-Kuhn-Tucker (KKT) conditions 
for \eqref{master}, which among other things specify the multiplier
$\mu^k$ employed in both \eqref{subp} and \eqref{gap}.

\begin{theorem}[KKT conditions for master QVI]
\label{kktqviset}
Under any suitable constraint qualification, if the pair $\left(x^k,z^k_m\in F(x^k)\right)$ solves the 
master QVI \eqref{master} at iteration $k$, 
then there exists a Lagrange multiplier $\mu^{k} \in \mathbb{R}^m$ such that
\begin{equation}\label{kkt1}
0 \in z_m^k+\sum\limits_{i=1}^m \mu^{k}_i\nabla_y g_i(x^{k},x^{k}) + 
\mathcal{N}_{K_h^k}(x^{k}),
\end{equation}
\begin{equation}\label{kkt2}
\mu_i^{k}\ge 0 ,\quad g_i(x^{k},x^{k}) \leq 0, \quad
\mu_i^{k} g_i(x^{k},x^{k}) =0, \; i=1,\ldots, m.
\hfill\qed\vspace{1em}
\end{equation}
\end{theorem}

We do not include a proof of this result, referring to
the corresponding theorem in \cite{Kanzow_2019}
in a more general setting of QVIs in Banach spaces,
under the Robinson's constraint qualification.
   For other suitable constraint qualifications in finite dimensions,
see \cite{mvs:CQs}.

\subsection{Gap function and finite termination} 

Convergence of Algorithm~\ref{algo1} relies  on the stopping
criterion $\GAP^k$ being
 driven to zero by the iterative process. For this reason, we
start by examining the properties of the gap function. 
Recall that according to Remark~\ref{monotonicity},
the mapping $F^k$ can always be chosen to be strongly monotone.

\begin{lemma}[Gap properties]\label{gap-lemma}
Let $F^k$ be strongly monotone.
At each iteration $k$ of Algorithm~\ref{algo1}, 
the following holds
for the gap function defined in \eqref{gap}: 
\begin{enumerate}
\item \label{le-gap}
$\GAP(y^{k+1})\leq 0$;
\item \label{ge-gap}
$\GAP(y)\geq 0 \mbox{ for all }y\in K_h^k$;
\item\label{z-gap} $\GAP(y^{k+1})=0$ if and only if $x^k=y^{k+1}$; and
\item\label{n-gap} if $\GAP(y^{k+1})<0$ then $y^{k+1}\not\in K_h^k$.
Hence, $K_h^k \subsetneq K_h^{k+1}$.
\end{enumerate}
\end{lemma}
\begin{proof}
Let $\left(y^{k+1},z_s^{k+1}\in F^k(y^{k+1})\right)$ be the pair 
solving the VI subproblem.  
The first inclusion from \eqref{atxk} states
that \(\myphi^k=z_m^k+\Gamma^k(x^k)\mu^k\in F^k(x^k) \), so
\begin{align}\label{strict}
z_s^{k+1}\in F^k(y^{k+1})\mbox{ and }
\myphi^k\in F^k(x^k) \Longrightarrow
\left< z_s^{k+1}-\myphi^k ,y^{k+1} -x^k\right> \geq  0\,,
\end{align}
by the monotonicity of $F^k$.                         
Recalling the gap definition \eqref{gap}, we obtain that
\[ \GAP(y^{k+1})= \left< \myphi^k ,y^{k+1}-x^k\right>\leq
\left< z_s^{k+1},y^{k+1}-x^k\right> \,.  \]
Because $x^k \in K_h$, 
the inequality in the first line in \eqref{subp} yields
\( \left< z_s^{k+1},x^k-y^{k+1}\right> \geq 0\).
We conlude that
\begin{equation}\label{gap-y} \GAP(y^{k+1})\leq
\left< z_s^{k+1},y^{k+1}-x^k\right> \leq0 ,  \end{equation}
proving item \ref{le-gap}.

To show item \ref{ge-gap},
let $\nu^{k} \in \mathcal{N}_{K_h^k}(x^{k})$ be the normal
element in Theorem~\ref{kktqviset} 
that makes the inclusion \eqref{kkt1} an equality:
\[ 0=z_m^{k}+\sum\limits_{i=1}^m\mu_i^{k}\nabla_yg_i(x^{k},x^{k})+\nu^{k}
=\myphi^k+\nu^k
\mbox{, with  }
\left< \nu^k,y-x^k \right> \leq 0\mbox{ for all }y\in K_h^k\,. 
\]
By definition \eqref{gap}, we have that
\( \GAP(y)=-\left<\nu^{k},y-x^k\right> \mbox{ for all }y\in\mathbb{R}^n\,.  \)
In particular, 
\[ \GAP(y)=-
\left< \nu^k,y-x^k \right> \geq 0\mbox{ for all }y\in K_h^k\,, 
\]
which proves item \ref{ge-gap}.

To continue with item \ref{z-gap}, suppose that
\(x^k=y^{k+1}\).
 Then 
$\GAP(y^{k+1})=\left<\myphi^k,y^{k+1}-x^k\right>=0$,
 recalling once again the gap definition in \eqref{gap}.
Conversely, when $\GAP(y^{k+1})= \left< \myphi^k ,y^{k+1}-x^k\right> =0$, we see from
\eqref{gap-y} that 
\( \left< z_s^{k+1},y^{k+1}-x^k\right> =0\).
In \eqref{strict} when have that
\( \left< z_s^{k+1}-\myphi^k ,y^{k+1}-x^k\right> =  0\),
which implies that   
$ x^k=y^{k+1}$, by strong monotonicity of $F^k$.

 For the final item,
for the sake of contradiction, assume that $y^{k+1}\in K_h^k$.
Then, by item \ref{ge-gap}, $\GAP(y^{k+1})\ge 0$.
This contradicts the hypothesis $\GAP(y^{k+1})<0$, concluding the proof
of item \ref{n-gap}.
\end{proof}

We can now show convergence when
the method terminates finitely.
Note that according to Lemma~\ref{gap-lemma} we have that 
$\GAP^{k}\le 0$  for all $k$, and therefore, 
Algorithm~\ref{algo1} stopping finitely at some iteration $k$  means
that $\GAP^{k+1}=0$ is computed.

\vspace{1em}
\begin{corollary}[Finite termination]\label{MsolvesVI}
If at some iteration $k$ of Algorithm~\ref{algo1}, 
$\GAP^{k+1}=0$ is computed, then the algorithm stops.
In this case, 
the pair 
 $\left(x^k,\widehat{\myphi}^k\right)$, with
$x^k$ solving the master problem \eqref{master}
and $\widehat{\myphi}^k$ from \eqref{atxk},
solves the original QVI,
given by \eqref{qVI} and \eqref{Kx}.
\end{corollary}
\begin{proof}
By item~\ref{z-gap} in Lemma~\ref{gap-lemma}, having
$\GAP(y^{k+1})=0$ is equivalent to $x^k = y^{k+1}$.
In particular, 
\[\begin{array}{ll}
x^k\in K_g(x^k)\cap K_h&\mbox{because $x^k$ solves \eqref{master} and \eqref{subp}, and}\\ 
\widehat{\myphi}^k\in F(x^k) &
\mbox{because
$\widehat\myphi^k\in\widehat F^k(y^{k+1})=
\widehat F^k(x^k)\subset F(x^k)$ for any choice in \eqref{aprox}.}
\end{array}\]
 To verify that the pair $\left(x^k,\widehat{\myphi}^k\in F(x^k)\right)$
solves \eqref{qVI}-\eqref{Kx}, it only remains
to show that
\begin{equation}\label{remain}
\left< \widehat{\myphi}^{k}, y-x^{k} \right> \geq 0\\
 \mbox{ for all }y \in K_g(x^{k})\cap K_h\,.
\end{equation}
Since 
$x^k$ solves the VI subproblem \eqref{subp}, it holds that
\[\left< z_s^{k+1}, y-x^k \right> \geq 0\; \mbox{ for all $y \in K_h$, in particular for all }
y \in K_g(x^{k})\cap K_h\,.\]
Then, as
\(z^{k+1}_s=\widehat\myphi^k+\Gamma^k(x^k)\mu^k\) 
by the definition of $\widehat{\myphi}^k$ in \eqref{atxk},
it holds that
\[\left< \widehat\myphi^k+\Gamma^k(x^k)\mu^k
, y-x^k \right> \geq 0\; \mbox{ for all }
y \in K_g(x^{k})\cap K_h\,.\]
The inequality \eqref{remain} will hold if
\begin{equation}\label{sol-qvi}
\left< \Gamma^k(x^k)\mu^k
, y-x^k \right> \leq 0\; \mbox{ for all }
y \in K_g(x^{k})\cap K_h\,.
\end{equation}
First recall that $\Gamma^k(x^k)\mu^k=\sum\limits_{i=1}^m \mu_i^k\nabla_yg_i(x^k,x^k)$.
Second, note that
because the functions $g_i(\cdot,x^k)$ are convex,
  for all $y \in\mathbb R^n$ we have that
\[\left< \nabla_y g_i(x^k,x^k), y-x^k\right> \leq
g_i(y,x^k)-g_i(x^k,x^k) \,.\]
Then multiplying this inequality by $\mu^k_i\geq 0$ 
and using the condition $\mu_i^kg_i(x^k,x^k)=0$ from
\eqref{kkt2} in Theorem~\ref{kktqviset} 
yields, for any \(y \in K_g(x^{k})\cap K_h\), that
\[\mu^k_i\left< \nabla_y g_i(x^k,x^k), y-x^k\right> \leq
\mu^k_ig_i(y,x^k)-\mu^k_ig_i(x^k,x^k)=
\mu^k_ig_i(y,x^k) \,.\]
Since $y\in K_g(x^k)$, the right-hand side in the relation above is nonpositive.
Summing up these inequalities for $i=1,\ldots,m$ yields \eqref{sol-qvi}, as claimed.
\end{proof}

\subsection{Asymptotic convergence of Algorithm~\ref{algo1}}
Corollary~\ref{MsolvesVI} states that
whenever the master QVI solution is also a solution to the VI subproblem, 
the algorithm stops having found a solution to the original QVI.
Otherwise, 
the iterative process continues and,
by item~\ref{n-gap} in Lemma~\ref{gap-lemma},
the algorithm makes progress by defining a larger
feasible set for the master QVI \eqref{master},
better approximating the original problem.
To complete our convergence analysis, it remains to consider
the asymptotic behavior of Algorithm~\ref{algo1}. Since
now we are dealing with an infinite number of iterations, 
the strong monotonicity of the approximations discussed in
Remark~\ref{monotonicity} needs to be ensured in a
uniform manner, for all iterates.	
Before stating the result, 
we note that
existence results about QVI solutions usually require 
some kind of monotonicity or continuity of $F$, assumptions
also used below.

\vspace{1em}

\begin{theorem}\label{principal}
Suppose that the operator $F$ in \eqref{qVI} 
is either continuous and single-valued  or
 maximally monotone set-valued, and that 
the function $g$ in \eqref{Kx} has a continuous gradient
$\nabla_y g (\cdot, \cdot)$.
 Let the VI-operator approximations $F^k$ in \eqref{subp} be defined
to be uniformly strongly monotone with 
parameter $c>0$, and let the matrices $\{Q^{k}\}$ be taken bounded. 
Finally, suppose that if Algorithm~\ref{algo1} 
generates an infinite number of itarations, 
the sequences $\{\mu^k\}$ and $\{y^{k+1}\}$ are bounded.

Then,
 $$\lim\limits_{k\to\infty}\GAP(y^{k+1}) = 0, \quad \lim\limits_{k\to \infty} \|y^{k+1}-x^k\| = 0,$$ 
and every accumulation point of $\left\{\left(x^k,\widehat\myphi^k\right)\right\}$ is a 
solution to QVI \eqref{qVI}-\eqref{Kx}.
\end{theorem}
\begin{proof}
By Lemma~\ref{gap-lemma} and Corollary~\ref{MsolvesVI},
for Algorithm~\ref{algo1} to make an infinite number of iterations,
it must hold that $\GAP(y^{k+1})<0$ for all $k$. 
Suppose the claim were not true, i.e.,
$\liminf_{k\to\infty}\GAP(y^{k+1}) < 0$. Then there
 exist $\varepsilon > 0$ and an infinite subset of iterations $\mathbb{N}_\varepsilon$, such that
$\GAP(y^{k+1})\leq-\varepsilon$ for all $k\in\mathbb{N}_\varepsilon$. 
Recalling \eqref{gap}, for this subsequence it holds that
\begin{align}
\label{gapep}
\left< z_m^{k} +\sum\limits_{i=1}^m\mu_i^{k}\nabla_yg_i(x^{k},x^{k}),y^{k+1}-x^{k}\right> 
\leq-\varepsilon\,.
\end{align}
Consider $k,j\in \mathbb{N}_\varepsilon$, with  ${j}>k$.
As $x^{j}$ solves the QVI master \eqref{master} at iteration $j$,
by \eqref{kkt1} in Theorem~\ref{kktqviset} there exists
$\nu^{j} \in \mathcal{N}_{\conv{Y^{{j}}}}(x^{j})$ such that
\begin{equation}
\label{auxkkt1}
0=z_m^{j}+\sum\limits_{i=1}^m\mu_i^{j}\nabla_yg_i(x^{j},x^{j})+\nu^{j}.
\end{equation}
By the construction of \eqref{master}, and since ${j}>k$, we have that $y^{k+1}\in \conv{Y^{{j}}}$, and hence,
\[\left< \nu^{j},y^{k+1}-x^{j}\right> \leq 0 .\]
Combining the latter relation with \eqref{auxkkt1}, we obtain that
\begin{equation}
\label{ineqj}
\left< z_m^{j}+\sum\limits_{i=1}^m\mu_i^{j}\nabla_yg_i(x^{j},x^{j}),y^{k+1}-x^{j}\right> \ge 0. 
\end{equation}

As $\{y^k\}$ is bounded, the sets $Y^k$ are uniformly bounded in $k$,
and hence so are the sets $K_g(\cdot)\cap K_h^k$ in \eqref{master}.
It follows that the sequence $\{x^k\}$ is bounded.
Let $\bar{x}$ be any accumulation point of the (bounded) subsequence
$\{x^{j}\}$, ${j}\in \mathbb{N}_\varepsilon$. Without ambiguity for further developments,
we can assume that the whole $\{x^{j}\}$, ${j}\in \mathbb{N}_\varepsilon$, converges to
some $\bar{x}$ (otherwise, just pass onto a subsequence within $\mathbb{N}_\varepsilon$,
and re-define $\mathbb{N}_\varepsilon$). 
We can then assume that $\{\mu^{j}\}$, ${j}\in \mathbb{N}_\varepsilon$, converges to some $\bar{\mu}$
(again, passing onto a further subsequence, if necessary).

If $F$ is  continuous single-valued, then
$z_m^{j} = F(x^{j})\to F(\bar{x}) = \bar{z}_m$, ${j} \in \mathbb{N}_\varepsilon$. 
If $F$ is a maximally monotone set-valued mapping, then
 it is locally bounded in $K_h$ because $K_h \subset \inte(\dom(F))$, and it is  
also outer-semicontinuous.
Then, again passing onto a further subsequence if necessary,
we can assume that
$z_m^{j} \to\bar{z}_m \in F(\bar{x}), {j} \in \mathbb{N}_\varepsilon.$ 

Then, taking the limit ${j}\to\infty, {j} \in \mathbb{N}_\varepsilon$, in \eqref{ineqj},
we conclude that
\begin{equation}
\label{subbar}
\left< \bar{z}_m+\sum\limits_{i=1}^m\bar{\mu}_i\nabla_yg_i(\bar{x},\bar{x}),
y^{k+1}-\bar{x}\right>\geq0, \quad \bar{z}_m \in F(\bar{x}) .
\end{equation}

Again, passing onto a subsequence if necessary, we can assume
that $\{y^{k+1}\}$, $k\in \mathbb{N}_\varepsilon$, converges to some $\bar{y}$.
Taking this limit in \eqref{gapep}, we obtain that
\[ \left<
\bar{z}_m+\sum\limits_{i=1}^m
\bar{\mu}_i\nabla_yg_i(\bar{x},\bar{x}),\bar{y}-\bar{x}\right>\leq-\varepsilon ,
\]
while taking the same limit in \eqref{subbar} yields
\[\left<
\bar{z}_m+\sum\limits_{i=1}^m
\bar{\mu}_i\nabla_yg_i(\bar{x},\bar{x}),\bar{y}-\bar{x}\right>\geq0.
\]           
This contradiction shows that $\liminf_{k\to\infty} \GAP(y^{k+1}) \ge 0$,
which means that $\lim_{k\to\infty} \GAP(y^{k+1}) = 0$, as $\GAP(y^{k+1})< 0$ for all $k$.

Since $y^{k+1}$ solves the VI subproblem \eqref{subp} at iteration $k$ and 
$x^k \in K_h$, it holds that  
\begin{align}
\label{desifor6e7}
\left< z_s^{k+1},x^k-y^{k+1}\right> \geq 0,
\end{align}
with $z_s^{k+1}\in F^k(y^{k+1})$ by \eqref{subp}.
As $F^{k}$ is uniformly strongly monotone 
(with modulus $c$), 
\[\left<
\myphi^k -z_s^{k+1} ,x^k-y^{k+1}\right> \geq c\, \|x^k-y^{k+1}\|^2, \,\]
because \( \myphi^k \in F^{k}(x^k) \text{ by \eqref{atxk}}.
\)
As a result, recalling \eqref{atxk}, we have that
\begin{equation}
    \label{strongVI}
\left<
z_m^k+\Gamma^k(x^k)\mu^k-z_s^{k+1},
 x^k-y^{k+1}\right> \geq c\, \|x^k-y^{k+1}\|^2 .
\end{equation}
Combining \eqref{strongVI} and \eqref{desifor6e7}, and recalling \eqref{gap}, 
we obtain that
\begin{align}
    \label{GAPbound}
-\GAP(y^{k+1}) \ge
c\, \|x^k-y^{k+1}\|^2.
\end{align}
Since $\GAP(y^{k+1})\to 0$ as $k\to\infty$ (as established previously), it follows that
$\|x^k-y^{k+1}\|\to 0$ as $k\to\infty$.

Let $(\bar{x},\bar{\myphi})$ be an accumulation point of $\{(x^k,\widehat\myphi^k)\}$.  
Since $x^k \in K_g(x^k)\cap K_h^k \subset K(x^k)$, and $g$ and $h$ are continuous, we have $\bar{x} \in K(\bar{x})$.
We can take a subsequence $\{k_j\}$ such that $x^{k_j}\to\bar{x}$, 
$\omega^{k_j} \to \bar{\omega}$, $\mu^{k_j} \to \bar{\mu}$ and 
$Q^{k_j}\to\bar{Q}$, as $j \to \infty$.

 As          
$$\|\bar{x}-y^{k_j+1}\|\leq\|\bar{x}-x^{k_j}\|+\|x^{k_j}-y^{k_j+1}\|,$$
taking the limit as $j \to \infty$, we obtain that
 $y^{k_j+1}\to \bar{x}$.

Let $y \in K(\bar{x})$ be arbitrary. It holds that $y \in K_h$.
As $(y^{k_j+1},z_s^{k_j+1}) \in K_h \times F^{k_j}(y^{k_j+1})$ solves 
the VI subproblem \eqref{subp} at iteration $k_j$, we have that
\begin{align} 
\left< z_s^{k_j+1},y-y^{k_j+1}\right> \geq 0\; \mbox{ and }\;
z_s^{k_j+1}\in F^{k_j}(y^{k_j+1}) .
\end{align}
Recalling the relation 
$\widehat\myphi^{k_j} = z_s^{k_j+1}-\Gamma^{k_j}(y^{k_j+1})\mu^{k_j}-Q^{k_j}(y^{k_j+1}-x^{k_j})$
from \eqref{atxk}, we obtain that
\begin{align}
\label{ineqmyphi}
\langle\widehat\myphi^{k_j}+
\Gamma^{k_j}(y^{k_j+1})\mu^{k_j}+Q^{k_j}(y^{k_j+1}-x^{k_j}),y-y^{k_j+1}\rangle \geq 0 .
\end{align}
Note that $\Gamma^{k_j}(y^{k_j+1})\mu^{k_j}$ tends to
$\sum\limits_{i=1}^m \bar{\mu}_i\nabla_yg_i(\bar{x},\bar{x})$ for any choice 
in \eqref{aprox-grad}, and $Q^{k_j}(y^{k_j+1}-x^{k_j})$ tends to zero,
as $j\to\infty$.
If $F$ is a continuous single-valued map, then for any choice of
$\widehat{F}^{k}$ in \eqref{aprox}, $\widehat\myphi^{k_j}=\widehat{F}^{k_j}(y^{k_j+1})$ tends to $\bar{\myphi} = F(\bar{x})$.  For the set-valued case, since $F$ is maximally monotone, 
it is outer semicontinuous and locally bounded. Then any accumulation point of any elements
in $\{\widehat{F}^{k_j}(y^{k_j+1})\}$ belongs to $F(\bar{x})$.
In either case, passing onto a subsequence in the second case if necessary, 
we have $\widehat\myphi^{k_j+1} \to \bar{\myphi} \in F(\bar{x})$. 
Taking the limit in \eqref{ineqmyphi}, we obtain that
$$\left< \bar{\myphi}+\sum\limits_{i=1}^m \bar{\mu}_i\nabla_yg_i(\bar{x},\bar{x}),
y-\bar{x}\right> \geq 0,  \text{ with }\bar{\myphi}\in F(\bar{x}).$$
Because 
$\bar{\mu}_i \ge 0$ and
$g_i(y,\bar{x})-g_i(\bar{x},\bar{x}) \ge 
\left<\nabla_yg_i(\bar{x},\bar{x}),y-\bar{x}\right>$
by convexity of $g(\cdot,\bar{x})$,  we have that 
\begin{eqnarray}\label{almost}
\left< \bar{\myphi},y-\bar{x}\right> &\geq&
 -\sum\limits_{i=1}^m \bar{\mu}_i\left<\nabla_yg_i(\bar{x},\bar{x}), 
y-\bar{x}\right> \nonumber \\
&\geq& -\sum\limits_{i=1}^m \bar{\mu}_i(g_i(y,\bar{x})-g_i(\bar{x},\bar{x})).
\end{eqnarray}
 By condition \eqref{kkt2} in Theorem~\ref{kktqviset}, 
$\mu_i^{k_j}g_i(x^{k_j},x^{k_j})=0$ and, hence,
 $\bar{\mu}_ig_i(\bar{x},\bar{x})=0$. 
Also, $\bar{\mu}_i\ge 0$ and, since $y \in K(\bar{x})$,  $\bar{\mu}_ig_i(y,\bar{x})\leq0$. 
From \eqref{almost} it follows that
$\left< \bar{\myphi}, y-\bar{x} \right> \geq 0,$ which shows
that $(\bar{x},\bar{\myphi}) \in K(\bar{x})\times F(\bar{x})$ solves  \eqref{qVI}-\eqref{Kx}, as stated.
\end{proof}

We next discuss
an important class of problems with
a certain block-separable structure.

\section{Block-Separable Approximations}\label{jacobi}

In many applications, including the one considered in Subsection~\ref{ss-walras},
the subproblem decision vector $y\in\mathbb R^n$ can be split into subvectors,
say, \[y=\prod\limits_{a\in A} y_a 
\mbox{ for }y_a\in \mathbb{R}^{n_a}\mbox{ and }\sum_{a\in A}n_a=n\,,\]
according to decomposable structures observed in the set $K_h$ in \eqref{Kx}.
Specifically, when
\[K_h=\prod\limits_{a\in A}K_{h_a}\mbox{ for }K_{h_a}=
\left\{y_a\in \mathbb{R}^{n_a}: h_a(y_a)\leq 0\right\}\,,\]
then the VI operator in subproblem \eqref{subp} 
\[F^{k}(y)=\widehat{F}^k(y)+\Gamma^k(y)\mu^k +Q^{k}(y-x^k)\,,\mbox{
for $\widehat F^k(y)$ from \eqref{aprox} and $\Gamma^k(y)\mu^k$ from \eqref{aprox-grad}},\] 
can be further decomposed as the product of VI-operators of smaller
dimensions:
\[F^{k}(y)= \prod\limits_{a\in A} \mathcal F^k_a(y_a)\mbox{ where }y=\prod\limits_{a\in A} y_a \,.\]
This is achieved by means of a Jacobi-like approach, similar to the one in \cite{Luna_2012}.
The process starts rearranging the $n$ components of $F^k$
according to the block structure:
\[F^k(y)=\prod\limits_{a\in A} F^k_a(y)\mbox{ where }y\in\mathbb R^n\,.\]
Then, each block $F^k_a(y)$ is approximated by an operator depending only on the $a$th subvector $y_a \in \mathbb{R}^{n_a}$,
fixing the remaining components to those of a known vector, for instance $x^k\in\mathbb R^n$:
\begin{equation}\label{fgj} 
F^k_a(y)\approx \mathcal F^k_{a}(y_a)= F^k_a(y_a,x^k_{-a})\quad\mbox{ where }
y_a\in\mathbb R^{n_a}\mbox{ and }
x^k_{-a}=\prod\limits_{a\neq j\in A} x^k_j \in\mathbb R^{n-n_a}\,.
\end{equation}
With this approximation,
\eqref{subp}  amounts to solving separate VI subproblems, each one of dimension $n_a$:
\begin{equation}\label{subpJ} \begin{array}{rlll}
\mbox{for }a\in A&\mbox{ find }(y_a^{k+1},z_{s_a}^{k+1})&\mbox{such that}
&y_a^{k+1} \in K_{h_a}\mbox{ and } z_{s_a}^{k+1} \in \mathcal F_a^k(y_a^{k+1}) \\
&&
\mbox{satisfy}
&
\left< z_{s_a}^{k+1}, y_a-y_a^{k+1} \right> \geq 0\; \mbox{ for all }y_a \in K_{h_a}\\
&& \mbox{where}&
\mathcal F_a^{k}(y_a)
\mbox{ is given in \eqref{fgj}.}\\
\end{array}
\end{equation}
To see that all the theoretical results in Section~\ref{convanalysis} remain valid
for this approxiomation,
notice first that evaluating each Jacobi block in \eqref{fgj} at $y_a=x^k_a$ gives the identity
\( F^k_a(x^k_a)= \mathcal F^k_a(x^k_a)\).
Taking the product over all the blocks,
\begin{equation}\label{eqs}
 F^k(x^k)= \mathcal F^k(x^k)
\mbox{ where }\mathcal F^k(y)=\prod\limits_{a\in A} \mathcal F^k_a(y_a)
\,.  \end{equation}
As the gap definition in \eqref{gap} remains the same, 
both Lemma~\ref{gap-lemma} and Corollary~\ref{MsolvesVI} hold.
Regarding the result on asymptotic convergence, a crucial property shown below is
that the Jacobi approximation preserves the monotonicity
of the original map. As a result, the (regularized) Jacobi approximation $\mathcal F^k$
can always be strongly monotone as long as $F$ is monotone.
Actually, Jacobi approximations need a weaker setting, referred 
to as block-wise monotonicity of $F$.
Specifically, monotonicity of the whole operator ${F}^k$ 
in Theorem~\ref{principal} can be replaced by requiring monotonicity
of the individual blocks $\mathcal{F}^k_a$.
This relaxed assumption can be useful in cases where $F$ is not monotone but has blocks which are monotone when
some components are fixed. Block-wise monotonicity occurs naturally, 
for example, in GNEPs. 

To simplify the notation, we illustrate these remarks for the single-valued maps only.

\begin{proposition}[Monotonicity properties of the Jacobi approximation] 
Consider Jacobi approximations as in \eqref{fgj}, defined for a single-valued operator $F^k$.

If for all $a\in A$ the Jacobi blocks 
$\mathcal F^k_{a}(y_a)$ are monotone with respect to $y_a\in\mathbb R^{n_a}$, then the full operator
$\mathcal F^k (y)=\prod\limits_{a\in A} \mathcal F^k_{a}(y_a)$ 
 is   monotone with respect to $y\in\mathbb R^n$ .
\end{proposition} \begin{proof} By definition of the full Jacobi approximation, we need to show that \[
\forall y, y'\in\mathbb{R}^{n},\quad 
\left< \mathcal F^k(y)-\mathcal F^k(y'),y-y'\right> =\sum_{a\in A}
\left< \mathcal F^k_{a}(y_a)-\mathcal F^k_{i}(y'_a),y_a-y'_a\right>\geq 0 \,.  \]
The result follows, because each term in the right hand side is non negative, by assumption.
\end{proof}

The statements in Theorem~\ref{principal} remain valid when replacing
the VI operator $F^k$ and subproblem \eqref{subp} by 
the Jacobi approximation $\mathcal F^k$ and \eqref{subpJ}. 
Most of the proof remains the same,
replacing throughout $F^{k}$ by $\mathcal F^k$
(for instance, using \eqref{eqs} 
in \eqref{strongVI} and \eqref{desifor6e7}, yields
\eqref{GAPbound}).

\section{Numerical Results}\label{numericalresults}
In order to assess the performance of our proposal, we apply Algorithm~\ref{algo1} to two
 well-known problems formulated as quasi-variational inequalities.
The first application, that aims at determining a stable state of equilibrium in an abstract economy, provides a good
setting to illustrate the interest of the Jacobi approximations presented in Section~\ref{jacobi}.
In the second application, an instance from \cite{QVILIBAL} called \emph{the moving-set problem}, 
we examine the impact of employing different approximations in \eqref{aprox-grad} in the VI subproblems \eqref{subp}.

The employed variants of the DW decomposition in each application
are compared against the direct solution of \eqref{qVI}-\eqref{Kx}, without decomposition. 
The two approaches are respectively referred to as {\sc dw} and {\sc direct}.

All of our runs were done on a computer running Ubuntu 22.04 with AMD Ryzen Threadripper 1950X processors, featuring 16 cores (32 threads) and 64GB of RAM.
The codes are available in \url{https://github.com/ManoelJardim/DWQVI/tree/main/DWQVI}.
The codes were written in Python, calling GAMS \cite{GAMS} and using its
Extended Mathematical Programming (EMP) tools.
The EMP tool was introduced in \cite{Kim_2019} to ease the GAMS formulation of 
specific equilibrium problems and their solution using PATH \cite{PATH}. 
We do not assign any parameters for GAMS and use its default values.

\subsection{Walrasian Equilibrium Problems}\label{ss-walras}
The model introduced in \cite{Arrow_1954}, stated here 
as in \cite{Facchinei_2010,Facchinei_2013}, is a generalized
Nash game involving firms, consumers, and market players. 
More precisely, there is one firm
producing $G$ goods that are bought by $C$ consumers, and a
player in charge of finding a price $p$ for the goods that clears the market
(having more firms is possible, but not essential for this application).
The Walrasian equilibrium prices the goods in
a manner that maximizes profit for the firm, utility for the
consumers, and clears the market to the best possible extent.

\subsubsection{Problem formulation and QVI blockwise structure}
On the demand side, letting $i=1,\ldots,C$, and given a price $p$,
the $i$th consumer
defines the goods to be bought, 
$x^i \in \mathbb{R}^G$, according to 
preferences determined by a concave utility function 
$\mathcal U^i(\cdot)$ and taking into account
an initial endowment $\mathcal{E}^i_j$ for each good
(that impacts on the consumer willingness to buy more of the good).
The $i$th consumer problem is
\begin{equation}\label{player}
\displaystyle\max_{x^i\geq 0}\left\{ \mathcal U^i(x^i) :
\displaystyle\left< p,x^i \right> \leq \left< p, \mathcal{E}^i\right> 
\right\}\,.
\end{equation}

The firm, denoted here as player $C+1$, decides its production of goods, $x^{C+1}\in \mathbb{R}^G$,
by solving the problem
\begin{equation}\label{firm}
\displaystyle\max_{x^{C+1}\geq 0} \left\{ \left< p, x^{C+1} \right>:
\sum\limits_{j=1}^G{(x^{C+1}_j)^2} \leq M
\right\}\,,
\end{equation}
where the capacity constraint depends on a parameter
$M>0$ and the price $p\in\mathbb{R}^G$ is given.

The equilibrium condition means that the supply meets exactly the demand in every market. Because the economy has $G$ goods and there are endowments,
the player in charge of clearing the market 
looks for (normalized) prices that solve the problem
\begin{equation}\label{market}
\displaystyle\max_{p\geq 0}\left\{
\displaystyle\left< p, \,\sum_{i=1}^C (x^i-\mathcal E^i) -x^{C+1}\right> 
:\displaystyle\sum_{j=1}^G p_j = 1 
\right\}.
\end{equation}

The Walrasian equilibrium 
results from reformulating 
the generalized Nash game 
associated with \eqref{player}, \eqref{firm}, and \eqref{market}
as a quasi-variational inequality \eqref{qVI}-\eqref{Kx}.
To this aim, we rename
the market-clearing player solving \eqref{market} as agent $C+2$
and its decision variable
$x^{C+2}=p$, so that the decision variable is
$x=\begin{bmatrix} x^1&\hdots&x^C&x^{C+1}&x^{C+2}\end{bmatrix}^T$. 

For $x\in\mathbb{R}^{(C+2)G}$, the game in the format \eqref{qVI}-\eqref{Kx},
has the operator  
\begin{equation}\label{W-F}
F(x) = 
\prod\limits_{i=1}^C \left(-\nabla_{x^i} \mathcal U^i(x^i)\right)\times 
\begin{bmatrix}
    -x^{C+2}\\
   \sum\limits_{i=1}^C (\mathcal{E}^i-x^i) +x^{C+1}\\
\end{bmatrix}\,,
\end{equation}
revealing a separable structure in its first $C$ components.
The constraint sets in \eqref{Kx} are
\begin{equation}\label{W-K}
\begin{array}{ll}
 K_g(x)&=\left\{y\in \mathbb{R}^{(C+2)G}:
g_i(y,x)=\sum\limits_{j=1}^Gx^{C+2}_j(y^i_j-\mathcal{E}^i_j)\leq 0,\mbox{ for } 1\leq i\leq C\right\}\,,
\\
 K_h&=\left\{y\in \mathbb{R}^{(C+2)G}:
\begin{bmatrix}-y\\\sum\limits_{j=1}^G{(y^{C+1}_j)^2-M}\end{bmatrix}\leq 0
\mbox{ and }
\sum\limits_{j=1}^G y^{C+2}_j -1=0
\right\}\,.\end{array}
\end{equation}
Since the easy constraints are uncoupled
for the different players, we separate $K_h$ following the
structure observed in \eqref{W-F}. Accordingly, we 
first define the sets
\[K_h^{a}= \left\{y^a\in \mathbb{R}^G: -y^a\leq0\right\}\,,
\quad\mbox{ if $1\leq a\leq C$.}\]
Constraints involving variables with index larger than $C$
are gathered into one set, indexed by $D$:
\[\begin{array}{rcl}K_h^{D}&=&\left\{y^D=(y^{C+1},y^{C+2})\in \mathbb{R}^G\times\mathbb R^G: -y^D\leq0\,,\right.\\
&&\hspace{3cm}\left.\sum\limits_{j=1}^G (y^{C+1}_j)^2 -M\leq 0\,,
\sum\limits_{j=1}^G y^{C+2}_j -1=0\right\}\,,\end{array}
\]
so that
\[
\mbox{ for } A=\{1,\ldots,C\}\cup\{D\}\,,\quad 
y=\prod_{a\in A} y^a\mbox{ and }
K_h=\prod_{a\in A} K^a\,.
\]
In this notation, the operator in \eqref{W-F} has the blockwise expression
\[
F(y) = \prod\limits_{a=1}^C F_a(y^a)\times F_{D}(y)
\mbox{ for }\left\{\begin{array}{rcl}
F_a(y^a)&=&-\nabla_{y^a} \mathcal U^a(y^a)\mbox{ if }1\leq a\leq C\,,\mbox{ and }\\
F_{D}(y)&=&
\begin{bmatrix}
    -y^{C+2}\\
    \sum\limits_{i=1}^C (\mathcal{E}^i-y^i) +y^{C+1}\\
\end{bmatrix}\,.
\end{array}\right.\]
To achieve separability in the VI subproblem operator, we make a Jacobi approximation
for $F_D$, fixing the consumer variables to their value at $x^k$:
\[{F^k_{D}}(y)\approx  \mathcal F^k_{D}(y^{D})=
{F^k_{D}}(x^k_1,\ldots,x^k_C,y^{C+1},y^{C+2})
\,.\]
Because of the linear dependence on $y$ for $g(y,x)$ in the current setting, 
any choice for $\Gamma^k$ 
in \eqref{aprox-grad} reduces to the constant approximation.
We take the exact option for $\widehat F^k$ in \eqref{aprox}.

\subsubsection{Solvers, data generation, and results}

For the {\sc direct} solver, 
we follow the recommendations in  
\cite[Section~3]{Kim_2019} to describe GNEPs using EMP.
The combination of EMP and GAMS was also employed in {\sc dw} to solve
all master QVI problems \eqref{master} and the VI subproblems \eqref{subpJ} 
for the block corresponding to $a=D$.
The remaining VI subproblems \eqref{subpJ}, the consumers' blocks $a=1,\ldots C$, are simple quadratic programs solved
calling the CVXOPT package in Python. 

To create random instances for \eqref{W-F}-\eqref{W-K}, we follow
\cite{Facchinei_2013, Facchinei_2010,QVILIBAL} and consider 
\[\mbox{quadratic concave utilities }
\mathcal U^i(x^i)=-\frac{1}{2}\left<x^{i },R^ix^i\right>+\left<b^i,x^i\right>
 \mbox{for }i=1,\ldots,C\,.\] 
The vector $b^i$ has $G$ uniformly distributed components ranging 
between 0 and 10. The positive semidefinite matrix $R^i$ is created as follows.
 First a random matrix $A^i$ of order $G$ is generated, with
elements uniformly distributed in $[-1,1]$. Then, for $B^i=A^{iT}A^i$, we 
set $R^i=\frac{10}{\|B^i\|_{\infty}}B^i$, so that
$R^i$ elements belong to $[-10,10]$ regardless the size of $G$. Endowments are randomly generated with a
uniform distribution in $[0,10]$. Finally, the firm's capacity $M$ was set at a sufficiently large value to meet the market demand in the large problems (the chosen value for $M$ had none or little impact in the results).

The numerical experiments were designed by varying the number of consumers $C$ and goods $G$, leading to 
QVI dimensions $n=(C+2)G$ ranging between 100 and 50,000. The results are grouped in 
two different benchmarks, noting that 20 different random instances were run for each considered dimension
in all the cases. For the {\sc dw} method, the initial point assigns prices equal to $\frac{1}{G}$ and zero to the remaining variables, whereas for the {\sc direct} method, we use the default GAMS parameters, which set the initial point to zero.

In {\sc benchmark} 1, we fix the number of consumers and vary the number of goods:
 $(C,G)\in \{20\}\times\{20,30,50,100,150,200,250\}$. Both solvers were run until triggering their
respective stopping test.
For this benchmark, and as a sanity test,
we also checked that the solutions found by {\sc direct} and {\sc dw} were numerically identical.
\begin{table}[hbt]
\caption{CPU time for solvers {\sc direct} and {\sc dw}, and number of {\sc dw} iterations with
140 instances of {\sc benchmark 1}. 
In each row, the solver in bold face is the one having the lowest mean time of execution.}\label{tab1}%
\begin{tabular}{@{}ccc|cccccc@{}}
\toprule\vspace{-.5em}
\multirow{2}{*}{\mbox{$n$}} &
\multirow{2}{*}{\mbox{$(C,G)$}}&
\multicolumn{1}{c|}{\mbox{ratio}}& 
\multicolumn{2}{l}{\mbox{{\sc direct} time (s)}} &
\multicolumn{2}{l}{\mbox{{\sc dw} time  (s)}} &
\multicolumn{2}{l}{\mbox{{\sc dw} iterations}}  \vspace{.5em}\\
\multicolumn{1}{l}{} &  &$G/C$ & mean & max  & mean & max  & mean & max\\
\midrule 
440&(20,20)   & 1  &\bf 0.17 & 0.19 & 1.51 & 2.78 & 9.75 & 16  \\
660&(20,30)   & 1.5 &\bf 0.27 & 0.29  & 1.49 & 2.28 & 9.30 & 12 \\
1,100&(20,50)   & 2.5 &\bf 0.59 & 0.65 & 1.41 & 2.65 & 7.50 &13\\
2,200&(20,100)    & 5 & 2.26 & 2.51 &\bf 1.21 & 2.54 & 5.45 & 10   \\
3,300&(20,150)    & 7.5 & 5.44 & 5.91 &\bf 1.30 & 1.80 & 4.05 & 5  \\
4,400&(20,200)    & 10 & 10.99 & 11.87 &\bf 1.32 & 2.37 & 4.25 & 7 \\
5,500&(20,250)    & 12.5 & 19.57 & 20.97 &\bf 1.54 & 1.94 & 4.05 & 5  \\
\botrule
\end{tabular}
\end{table}

As seen in the output reported in Table~\ref{tab1}, 
{\sc direct} performs with almost instantaneous execution times
for dimensions up to 1,100. 
But as the problem size increases, {\sc dw} becomes competitive and largely surpasses {\sc direct}. 
The benchmark confirms the positive impact of decomposition: 
the sequential {\sc dw} iterative process, solving 
small QVI master problems and decoupled VI-subproblems, finds the same solutions as {\sc direct} 
with eventually less computational effort.
Table~\ref{tab1} also hints at an interesting phenomenon, related to the values reported
in the third column. Specifically, {\sc dw} surpasses {\sc direct} when
the ratio $G/C$ is larger than 2.5. Moreover for $G/C\geq 5$ the number of {\sc dw} iterations and its
mean running time remains practically the same, even though the dimension $n$ grows from 2,200 to 5,500.

\begin{figure}[htb]
\centering
\includegraphics[width=1\textwidth]{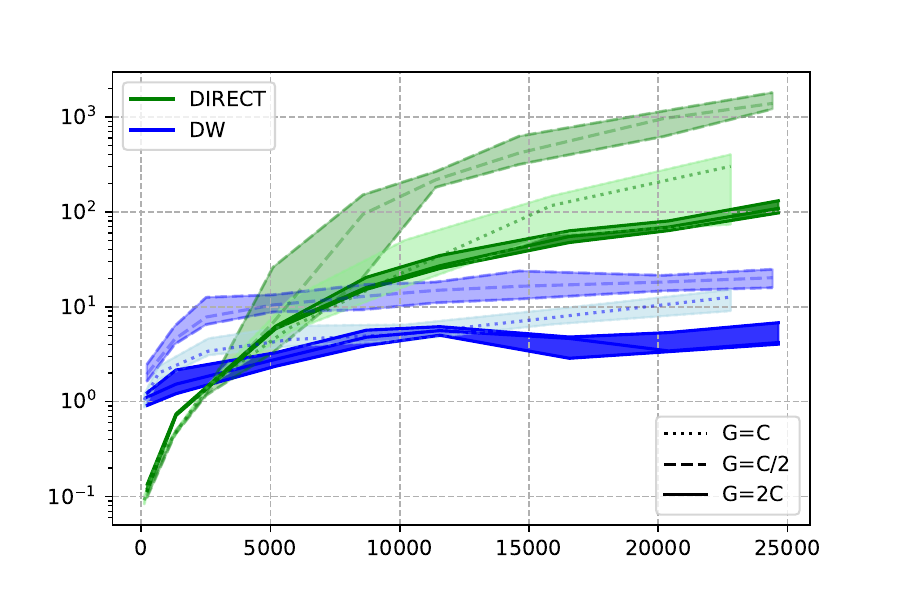}
\caption{Time statistics (median and 25\%-75\% quantiles) 
for {\sc direct} and {\sc dw}, separating the three different ratios between goods and consumers in Table~\ref{tab2}.
For better illustration, three different economy configurations are plotted separately, in a semilogarithmic 
scale and barring the last two lines in the table.
}\label{fig1}
\end{figure}

In order to include more configurations favorable to the direct solution,
the next benchmark varies the ratio $C/G$ in $\{0.5,1,2\}$, generating instances with QVI dimension $n\in[120,40,400]$.
Statistics for the corresponding results are reported graphically in
Figure~\ref{fig1}, with full details in  Table~\ref{tab2}.

Figure~\ref{fig1} gives a clear graphical confirmation that, as $n$ increases, {\sc dw} solving
time remains stable while {\sc direct}'s exhibits an exponential growth. 
Also, the wide shaded areas observed for {\sc direct} when $G=C$ and $G=C/2$ correspond to
the direct solution approach being more sensitive than {\sc dw} to the ration $G/C$.
In the detailed output in Table~\ref{tab2}, for different configurations of the economy {\sc dw} performs
better than {\sc direct} as soon as the QVI dimension $n=(C+2)G$ becomes larger than 5,000.

\begin{table}[htb]
\caption{CPU time for solvers {\sc direct} and {\sc dw}, and number of {\sc dw} iterations with
540 instances of {\sc benchmark 2}. 
In each row, the solver in bold face is the one having the lowest mean time of execution.}\label{tab2}%
\begin{tabular}{@{}ccccccccc@{}}
\toprule
\multirow{2}{*}{\mbox{$n$}} &
\multirow{2}{*}{\mbox{$(C,G)$}}&
\multirow{2}{*}{\mbox{ratio}}&
\multicolumn{2}{l}{\mbox{{\sc direct} time (s)}} &
\multicolumn{2}{l}{\mbox{{\sc dw} time (s)}} &
\multicolumn{2}{l}{\mbox{{\sc dw} iterations}} \vspace{.5em}\\
\multicolumn{1}{l}{} & &$G/C$ & mean & max & mean & max & mean & max\\
\midrule
120 & (10, 10) & 1 & \bf 0.09 & 0.14 & 1.04 & 1.86 & 7.80 & 13 \\
220 & (20, 10) & 0.5 &\bf  0.11 & 0.16 & 2.12 & 4.06 & 12.90 & 22 \\
240 & (10, 20) & 2 &\bf  0.13 & 0.15 & 1.11 & 2.02 & 8.30 & 14 \\
625 & (25, 25) & 1 &\bf  0.25 & 0.29 & 2.10 & 3.21 & 11.95 & 17 \\
1,300 & (50, 25) & 0.5 &\bf  0.47 & 0.60 & 4.99 & 9.80 & 19.05 & 30 \\
1,350 & (25, 50) & 2 &\bf  0.73 & 0.78 & 1.67 & 3.28 & 8.05 & 14 \\
2,520 & (70, 35) & 0.5 &\bf  1.27 & 2.00 & 10.09 & 21.60 & 23.55 & 38 \\
2,600 & (50, 50) & 1 &\bf  1.49 & 1.61 & 3.79 & 5.69 & 13.15 & 18\\
5,100 & (100, 50) & 0.5 & 15.81 & 65.71 &\bf  11.53 & 26.89 & 21.25 & 36\\
5,200 & (50, 100) & 2 & 6.10 & 6.84 &\bf  2.96 & 6.56 & 8.20 & 16 \\
5,775 & (75, 75) & 1 & 8.18 & 14.21 &\bf  5.00 & 8.46 & 11.75 & 18 \\
8,580 & (130, 65) & 0.5 & 107.90 & 327.86 &\bf  13.29 & 23.29 & 18.95 & 28 \\
8,710 & (65, 130) & 2 & 19.73 & 45.98 &\bf  4.76 & 8.08 & 6.95 & 11 \\
10,200 & (100, 100) & 1 & 84.27 & 638.98 &\bf  5.96 & 11.25 & 10.20 & 17 \\
11,400 & (150, 75) & 0.5 & 260.89 & 749.89 &\bf  15.21 & 29.71 & 18.35 & 29 \\
11,550 & (75, 150) & 2 & 31.73 & 82.83 &\bf  5.74 & 9.00 & 6.70 & 10 \\
14,620 & (170, 85) & 0.5 & 513.31 & 1,499.38 &\bf  17.73 & 31.99 & 18.25 & 28 \\
15,875 & (125, 125) & 1 & 438.68 & 2,903.90 &\bf  7.81 & 11.60 & 9.30 & 13\\
16,560 & (90, 180) & 2 & 68.37 & 195.74 &\bf  4.18 & 8.38 & 5.40 & 10 \\
20,200 & (200, 100) & 0.5 & 953.97 & 2,420.30 &\bf  18.08 & 31.22 & 15.85 & 24 \\
20,400 & (100, 200) & 2 & 75.97 & 113.51 &\bf  4.34 & 9.13 & 4.90 & 9 \\
22,800 & (150, 150) & 1 & 450.60 & 2,331.93 &\bf  13.01 & 24.37 & 8.15 & 14 \\
24,420 & (220, 110) & 0.5 & 1,638.28 & 4,765.29 &\bf  21.32 & 37.76 & 15.80 & 24 \\
24,640 & (110, 220) & 2 & 156.74 & 604.68 &\bf  5.24 & 8.32 & 4.85 & 7 \\
30,975 & (175, 175) & 1 & 583.40 & 5,192.05 &\bf  9.23 & 21.55 & 6.25 & 13 \\
31,750 & (125, 250) & 2 & 179.81 & 306.53 &\bf  5.31 & 5.47 & 4.00 & 4 \\
40,400 & (200, 200) & 1 & 1,186.40 & 9,133.32 &\bf  10.12 & 22.72 & 6.10 & 12 \\
\botrule
\end{tabular}
\end{table}

For both solvers, the lower the ratio $G/C$, the harder the problem becomes,
because of the number of constraints.
In all the considered instances, {\sc dw} took less than 20 seconds in average, while 
{\sc direct} times varied between 0.09 seconds and 30 minutes.
As hinted by the shaded areas in Figure~\ref{fig1} and noticeable in 
the fifth column in Table~\ref{tab2}, {\sc direct}'s maximum execution times
appear to be highly sensitive to the economy configuration. 

\subsection{Moving Set Problems}\label{mov}
We now turn our attention to problems MovSet3A1, MovSet3B1, MovSet3A2, and MovSet3B2 from 
\cite{QVILIBAL}, see also \cite{Facchinei_2013}. In these problems,  there is no $h$ constraint
and $g(x,y)$ depends quadratically on $y$.

Given positive dinite matrices $A,R \in \mathbb{R}^{n\times n}$, 
$B \in \mathbb{R}^{n\times n}, b\in \mathbb{R}^n$, and $d \in \mathbb{R}$, 
the QVI operator in \eqref{qVI} and constraint sets in \eqref{Kx} are
\[
\begin{array}{l}
F(x)=Ax+b\mbox{ for }x\in\mathbb R^n\,,\\
 K_g(x)=\{y\in \mathbb{R}^n: g(y,x) =\left< R(y-Bx), y-Bx\right> -d\leq 0\}\mbox{, and }
 K_h=\mathbb{R}^n\,.\end{array}
\]

As $\nabla_y g(y,x)=2R(y-Bx)$, the setting is suitable to analyze the impact of different values of $\omega$
for $\Gamma^k_j(y)$ in \eqref{aprox-grad}. 
Since $F$ is a linear operator, we take the (exact) approximation $\widehat F^k(y)=F(y)$ in \eqref{aprox}, which
yields, for $\omega^k\in[0,1]$ and $\mu^k\geq0$, 
the following VI operator:
\[
F^{k}(y) = Ay+b+\Bigl((1-\omega^{k})2(Ry-Bx^k)+\omega^k2(R-B)x^k\Bigr)\mu^k\,.
\]
As $K_h=\mathbb R^n$, in this case subproblems \eqref{subp} are simple linear systems
 of equations in $y$.

\begin{table}[hbt]
\caption{MovSet problems using different $\omega$ in $\Gamma^k$ for DW.}\label{MovSetTable}%
\begin{tabular}{@{}cccccc@{}}
\toprule
Problem & n & {\sc direct} (s)& DW (s) & DW (s) & DW (s) \\  & & & $\omega=1$ & $\omega=0$ & $\omega=0.5$ \\
\midrule 
MovSet3A1 & 1,000 & 51.97 & 1.71 (11 it.) &\bf 0.74 (5 it.) & 0.83 (6 it.) \\
MovSet3B1 & 1,000 & 50.88 & 1.64 (11 it.) &\bf 0.72 (5 it.) & 0.90 (6 it.)\\
MovSet3A2 & 2,000 & 285.02 & 4.60 (12 it.) &\bf 1.90 (5 it.) & 3.00 (7 it.)\\
MovSet3B2 & 2,000 & 284.24 & 4.44 (11 it.) &\bf 1.75 (5 it.) & 1.92 (6 it.)\\
\botrule
\end{tabular}
\end{table}
The comparison with {\sc direct} is reported in Table~\ref{MovSetTable}, 
indicating with bold face the
fastest solver for each run. For both methods, the initial point is set to zero.
For this test set, subproblems are solved easily and it is best to include
the full information of the relaxed constraint $g$ ($\omega=0$, i.e.,
 the free option in \eqref{aprox-grad}).

\if{
\section{Conclusion}\label{sec13}

We have proposed a family of decomposition methods for quasi-variational inequalities
 inspired by the Dantzig-Wolfe technique. The method performs well for large-scale problems, 
as indicated by the presented numerical results 
for Walrasian equilibrium problems and moving-set problems.

 For future research, other decomposition methods may be explored for QVIs and other equilibrium problems, 
possibly by weakening the assumptions and using different approximation techniques.
}\fi

\backmatter

\section*{Declarations}

\noindent {\small
{\bf Conflicts of Interest.}
The authors declare that they have no conflict of interest of any kind related to the manuscript.}

\noindent {\small
{\bf Data Availability Statement.}
Data sharing is not applicable to this article.}


\bmhead{Acknowledgements}

The second author is supported by CNPq Grant 307509/2023-0 and by PRONEX--Optimization.
The third author is supported in part by CNPq Grant 306775/2023-9,
by FAPERJ Grant E-26/200.347/2023, and by PRONEX--Optimization.

The authors thank Professor Paulo~J.S.~Silva from UNICAMP for facilitating the access
to IMECC's computational resources that were necessary for conducting our numerical experiments.

\bibliography{references.bib} 

\end{document}